\documentclass{amsart}
\usepackage{dsfont}
\usepackage{mathrsfs}
\usepackage{color}
\usepackage{amssymb,mathrsfs,amsmath,amsthm,color,bm,mathtools,bbm,wasysym,subfig,float}
\usepackage{CJK}
\usepackage{amssymb,bm}
\usepackage[noadjust]{cite}
\usepackage[centering]{geometry}
\allowdisplaybreaks
\newtheorem{theorem}{Theorem}[section]
\newtheorem{lemma}[theorem]{Lemma}
\theoremstyle{definition}

\newtheorem{proposition}{Proposition}[section]

\theoremstyle{remark}

\newtheorem{remark}{Remark}[section]

\numberwithin{equation}{section}

\makeatletter

\newcommand{\Rmnum}[1]{\expandafter\@slowromancap\romannumeral #1@}
\makeatother

\begin{document}

\title{Metric Theory for Continued Fractions with Multiple Large Partial Quotients}


\author{Qian Xiao*}
\address{School of Mathematics and Statistics, Southwest University, Chongqing, 400715, P.R. China}
\email{xiaoqianmath@163.com}

\thanks{* Corresponding author}





\begin{abstract}
The presence of large partial quotients can invalidate many classical limit theorems in the metric theory of continued fractions. A commonly employed strategy to overcome this problem is to discard the largest partial quotient when formulating variant forms of such theorems.
However, this method will fail when dealing with at least two large partial quotients.
Motivated by recent work of Tan, Tian, and Wang [Sci. China Math., 2023], we investigate the metric theory of real numbers that contain at least $r$ large partial quotients among the first $n$ terms of their continued fraction expansions.
Specifically, let $[a_1(x), a_2(x), \ldots ]$ be the continued fraction expansion of a real number $x \in [0, 1)$.
We determine the Lebesgue measure and Hausdorff dimension of the following set:
\[
F(r, \psi)=\Big\{ x \in [0,1): \exists 1 \leq k_1< \cdots < k_r \leq n, a_{k_i} (x) \geq \psi(n)~(i=1, \ldots, r)\  \text{for i.m.}~n \in \mathbb{N} \Big\},
\]
where `i.m.' stands for `infinitely many', $r \geq 1$ and $\psi$ is a positive function defined on $\mathbb{N}$.
\end{abstract}

\maketitle

\section{ \emph{Introduction}}

It is well known that every real number $x \in [0,1)$ admits a continued fraction expansion of the form
\begin{align}\label{CE}
	x  &=\frac{1}{a_1(x) +\cfrac{1}{a_2(x) +\cfrac{1}{a_3(x) +\ddots}}}  :=[a_1(x), a_2(x), a_3(x), \ldots],
\end{align}
where $a_1(x), a_2(x), a_3(x), \ldots$ are positive integers, known as the $\mathbf{partial~quotients}$ of $x$.
The expansion is generated by the Gauss map $T: [0, 1) \to [0, 1)$ defined by
\begin{equation}\label{def:Gauss}
	T(0) =0, ~ T(x) = 1/x - \lfloor 1/x \rfloor  ~\text{if}~x \in (0,1),
\end{equation}
where $\lfloor 1/x \rfloor$ denotes the integer part of $1/x$.
For any $n \geq 1$, the finite truncation
\[
[a_1(x), a_2(x), \ldots, a_n(x)]: = \frac{p_n(x)}{q_n(x)}
\]
is called the $n$th $\mathbf{convergent}$ of $x$.

The metric theory of continued fractions is concerned with the quantitative study of the properties of partial quotients for almost all $x \in [0,1)$.
One of its central topics is the study of the metric properties of the classical set of $\psi$-well approximable numbers
\begin{equation}
	\bigg\{x \in [0,1): a_{n+1}(x) \geq \frac{1}{q_n(x)^2 \psi(q_n(x))} \  \ \text{for i.m.}~n \in \mathbb{N} \bigg\}.
\end{equation}
Here, `i.m.' stands for `infinitely many' and $\psi: \mathbb{N} \to \mathbb{R}$ is a positive non-increasing function.
The classical theorems of Khintchine and Jarn\'ik in Diophantine approximation were both established through the study of this set.
A real number $x$ is said to be $\psi$-$\mathbf{well~approximable}$ if its partial quotients grow sufficiently fast.
In other words, the growth rate of the partial quotients reveals how well a real number can be approximated by rational numbers.
Let $\psi : \mathbb{N} \to \mathbb{R}^{+}$ be a positive function that tends to infinity as $n \to \infty$,
and define
\[
E_1(\psi): =\Big\{x \in [0,1): a_n(x) \geq \psi(n)\  \ \text{for i.m.}~n \in \mathbb{N} \Big\}.
\]
The well-known Borel--Bernstein theorem \cite{Bernstein1911, Borel1912}, an analogue of Borel--Cantelli lemma with respect to Lebesgue measure for the set of real numbers with large partial quotients, describes the Lebesgue measure of $E_1(\psi)$ as follows:
\begin{theorem}[Borel--Bernstein]
	The Lebesgue measure of $E_1(\psi)$ is either zero or full according as the series $\sum_{n=1}^{\infty} 1/ \psi(n)$ converges or diverges, respectively.
\end{theorem}

Indeed, Wang and Wu \cite{Wang2008} provided a complete characterization of the Hausdorff dimension of the set $E_1(\psi)$, stated as follows:
\begin{theorem}[{\cite[Theorem 4.2]{Wang2008}}]
	Let $\psi$ be an arbitrary positive function defined on natural numbers. Define
	\[
	\log B = \liminf_{n \to \infty} \frac{\log \psi(n)}{n} \quad
	\text{and} \quad
	\log b =\liminf_{n \to \infty} \frac{\log \log \psi(n)}{n}.
	\]
	\begin{itemize}
		\item[(1)] If $B=1$, then $\dim_{\mathrm{H}} E_1(\psi) =1$.
		\item[(2)] If $1 < B < \infty$, then
		\[
		\dim_{\mathrm{H}} E_1(\psi)= \inf \Big\{s \geq 0: P(T, -s\log B-s \log |T'| )\leq 0 \Big\}.
		\]
		\item[(3)] If $B=\infty$, then $\dim_{\mathrm{H}} E_1(\psi)=\frac{1}{1+b}$.
	\end{itemize}
	Here $P$ denotes the pressure function, defined in Section 2.3.
	 Throughout this paper, $\dim_{\mathrm{H}}$ stands for the Hausdorff dimension and $T'$ denotes the derivative of the Gauss map $T$.
\end{theorem}

On the one hand, motivated by the study of uniform Diophantine approximation, Kleinbock and Wadleigh \cite{Kleinbock2018} introduced the following set
\[
E_2(\psi)= \Big\{x \in [0, 1): a_n(x)  a_{n+1}(x) \geq \psi(n)~\text{for~i.m.}~n \in \mathbb{N} \Big\}.
\]
They showed that the Lebesgue measure $\mathcal{L}(E_2(\psi))$ is zero or full depending on whether the series $\sum_{n=1}^{\infty} \log \psi(n)/ \psi(n)$ converges or diverges, respectively.
Furthermore, for an integer $m \geq 1$, define the following set
\[
E_m(\psi)= \Big\{x \in [0, 1): a_n(x) \cdots a_{n+m-1}(x) \geq\psi(n) ~\text{for i.m.}~n \in \mathbb{N} \Big\}.
\]
In \cite{Huang20}, Huang, Wu and Xu gave a complete description of  the Lebesgue measure of the above set.
\begin{theorem}[{\cite[Theorem 1.5]{Huang20}}]
	Let $\psi : \mathbb{N} \to [2, +\infty)$ be a positive function. Then
	\begin{equation*}
		\mathcal{L}(E_m(\psi))=\begin{dcases}
			0,   & ~\text{if}~\sum\limits_{n =1}^{\infty} \frac{\log^{m-1} \psi(n)}{\psi(n)} < \infty,\\
			1,   & ~\text{if}~\sum\limits_{n = 1}^{\infty } \frac{\log^{m-1} \psi(n)}{\psi(n)} = \infty. \\
		\end{dcases}
	\end{equation*}
\end{theorem}

In addition, Huang, Wu and Xu determined the Hausdorff dimension of the set  $E_m (\psi)$:
\begin{theorem}[{\cite[Theorem 1.7]{Huang20}}]
	Write
	\[
	\log B = \liminf_{n \to \infty} \frac{\log \psi(n)}{n} \quad
	\text{and} \quad
	\log b =\liminf_{n \to \infty} \frac{\log \log \psi(n)}{n}.
	\]
	Then
	\begin{equation}
	\dim_{\mathrm{H}} E_m (\psi)=\begin{dcases}
		1,   & ~\text{if}~B=1,\\
		\inf\Big\{s \geq 0: P(T, -f_m(s) \log B-s \log |T'| ) \leq 0 \Big\},   & ~\text{if}~1< B< \infty, \\
		\frac{1}{1+b}, &\text{if}~B=\infty,
	\end{dcases}
	\end{equation}
	where $f_m(s)$ is given by the following iterative formula:
	\[
	f_1(s)=s, \quad f_{k+1}(s) = \frac{s f_k (s)}{1-s+f_k(s)}, k \geq 1.
	\]
\end{theorem}

For further studies on related sets, we refer readers to \cite{Hussain23, Tan24}.
Clearly, understanding the growth rate of the product of $m$ consecutive partial quotients plays a pivotal role in analyzing the set of Dirichlet non-improvable numbers.

On the other hand, the existence of large partial quotients also destroys some limit theorems in continued fractions, and there is no finite law of large number, i.e.,
\[
\lim_{n \to \infty} \frac{a_1(x)+a_2(x) + \cdots +a_n(x)}{n}  = \infty~\text{almost~surely}.
\]
Let $S_n (x) = \sum_{i=1}^{n} a_i(x)$,  and Khintchine \cite{Khintchine63} proved that $\frac{S_n(x)}{n \log n}$ converges to $\frac{1}{\log 2}$.
However, it is shown by Philipp \cite{Philipp88} that there is no suitable normalizer replacing of the denominator $n$ such that the above limit is positive and finite.
 Diamond and Vaaler \cite{Diamond1986} proved that if the largest $a_i(x)$ is removed from $S_n(x)$,
then for almost every $x \in [0, 1)$,
\[
\lim_{n \to \infty} \frac{S_n(x)-\max \{ a_i(x): 1 \leq i \leq n\} }{n \log n}  =\frac{1}{\log 2}.
\]
A crucial step in obtaining the above result is to show that, for sufficiently large $n$, the set of numbers with two large partial quotients among the first $n$ terms is rare.
More precisely, define
\[
F(2, \psi) = \Big\{x \in [0, 1): \exists~ 1 \leq k \neq l  \leq n, a_k(x) \geq \psi(n), a_l(x) \geq \psi(n)~\text{for i.m.}~n \in \mathbb{N} \Big\}.
\]
Diamond and Vaaler \cite{Diamond1986} showed that $\mathcal{L}\big(F(2, \psi)\big)=0$ when $\psi(n)= n (\log n)^c$ with $c > 1/2$.
However, their result provides only partial information about the size of $F(2, \psi)$.
In comparison with the classical Borel--Bernstein theorem, Tan, Tian and Wang \cite{Tan23} established the following dichotomy law for the Lebesgue measure of $F(2, \psi)$:
\begin{theorem}[{\cite[Theorem 1.3]{Tan23}}]
	Let $\psi : \mathbb{N} \to \mathbb{R}^+$ be a non-decreasing function. Then
	\begin{equation*}
	\mathcal{L}\big(F(2, \psi)\big)=\begin{dcases}
		0,   & ~\text{if}~\sum_{n =1}^{\infty} n \big(\psi(n) \big)^{-2} < \infty,\\
		1,   & ~\text{if}~ \sum\limits_{n = 1}^{\infty } n \big(\psi(n)\big)^{-2} = \infty. \\
	\end{dcases}
	\end{equation*}
\end{theorem}
Moreover, they determined the Hausdorff dimension of $F(2, \psi)$ as follows:
\begin{theorem}[{\cite[Theorem 1.3]{Tan23}}]
	Let $\psi : \mathbb{N} \to \mathbb{R}^+$ be a non-decreasing function and
	\[
		\log B = \liminf_{n \to \infty} \frac{\log \psi(n)}{n} \quad
	\log b =\liminf_{n \to \infty} \frac{\log \log \psi(n)}{n}.
	\]
	We see that
	\begin{equation*}
	\dim_{\mathrm{H}} F(2, \psi)=\begin{dcases}
				1,   & ~\text{if}~B=1,\\
				\inf\Big\{s \geq 0: P(T, -(3s-1) \log B-s \log |T'| ) \leq 0 \Big\},   & ~\text{if}~1< B< \infty, \\
				\frac{1}{1+b}, &\text{if}~B=\infty.
			\end{dcases}
	\end{equation*}
\end{theorem}

Motivated by the preceding discussion, we investigate the set of points for which there are at least $r$ large partial quotients among the first $n$ terms for infinitely many sufficiently large $n$.
To be more precise,
for any $x \in [0,1)$, write
\[
L_n(x, \psi(n)) := \# \Big\{ k \leq n: a_k(x) \geq \psi(n) \Big\},
\]
where $\# A$ denotes the cardinality of the set $A$.
The quantity $L_n(x, \psi(n))$ counts the number of partial quotients among the first $n$ terms of $x$  that are greater than $\psi(n)$.
Fix an integer $r \geq 1$ and define
\[
F(r, \psi): =\Big\{ x \in [0,1): L_n(x, \psi(n)) \geq r\  \text{for i.m.}~n \in \mathbb{N} \Big\},
\]
or equivalently rewritting as
\[
\Big\{ x \in [0,1): \exists 1 \leq k_1< \cdots < k_r \leq n, a_{k_i} (x) \geq \psi(n)~(i=1, \ldots, r)\  \text{for i.m.}~n \in \mathbb{N} \Big\}.
\]
We provide a characterisation on the size of $F(r, \psi)$ in terms of both Lebesgue measure and Hausdorff dimension.

\begin{theorem}\label{leb}
Let $\psi$ be an arbitrary positive function defined on natural numbers.
Then
\begin{equation*}
  \mathcal{L}\big(F(r, \psi)\big)=\begin{dcases}
  		0,   & ~\text{if}~  \sum_{n=1}^{\infty} n^{r-1} \Big(\min_{m\geq n}\psi(m)\Big)^{-r} < \infty,\\
  	1,   & ~\text{if}~  \sum_{n=1}^{\infty} n^{r-1} \Big(\min_{m\geq n}\psi(m)\Big)^{-r} = \infty. \\
  \end{dcases}
\end{equation*}
\end{theorem}

We also determine the Hausdorff dimension of the set $F(r, \psi)$:
\begin{theorem}\label{Hausdorff}
	Let $\psi$ be an arbitrary positive function defined on natural numbers.
	Write
\begin{equation*}
		\log B: = \liminf_{n \to \infty} \frac{\log \Big(\min_{m\geq n}\psi(m)\Big) }{n}~ \text{and}
	~\log b :=  \liminf_{n \to \infty} \frac{\log \log \Big(\min_{m\geq n}\psi(m)\Big)}{n}.
\end{equation*}
	Then
	\begin{equation*}
		\dim_{\mathrm{H}} F(r, \psi)=\begin{cases}
			1,  &\text{ if } B=1,\\
			\inf \Big\{ s \geq 0: P\Big(T, -s \log |T'|-\big(s+(2s-1)(r-1)\big) \log B \Big) \leq 0 \Big\},   & \text{ if } 1 < B < \infty,  \\
				\dfrac{1}{1+b},   &\text{ if }  B= \infty.
		\end{cases}
	\end{equation*}
\end{theorem}

\begin{remark}
(1) Compared to the results in \cite{Tan23}, our setting does not require the function $\psi$ to be monotone.
In fact, let $\widetilde{\psi}(n) = \min_{m \geq n} \psi(m)$.
Then $\widetilde{\psi}$ is a non-decreasing function, and by Proposition \ref{p:mon}, the set $F(r, \psi)$ remains unchanged if we replace $\psi$ with $\widetilde{\psi}$.

(2q) Regarding the measure-theoretic results, although Tan, Tian, and Wang \cite{Tan23} noted in Remark 1.5 that their results can be extended to the case when $r > 2$, they did not provide an explicit 0-1 measure criterion for this setting. Moreover, their formula (see Theorem 1.5) does not clearly indicate how the exponent of $n$ depends on $r$. In contrast, our approach not only establishes precise criteria for all $r \geq 2$ (see Theorem \ref{leb}), but also employs an independent method.
\end{remark}

The structure of the paper is as follows. Section 2 introduces the necessary definitions and some preliminary results.
The proofs of Theorems \ref{leb} and \ref{Hausdorff} are presented in Sections 3 and 4, respectively.

\section{\emph{Preliminaries}}

In this section, we recall some basic properties of continued fractions and establish some basic facts.

Throughout this paper, we list the following notation:
\begin{itemize}
	\item a sequence of integers will be denoted by letters in boldface: $\bm a, \bm b, \ldots$,
	\item for two variables  $\alpha$ and $\beta$, the notation $\alpha \lesssim \beta$ means that $\beta \leq c \alpha $ for some unspecified constant $c$, and the notation $\alpha \asymp \beta $  means that $\alpha \lesssim \beta$ and $\beta \lesssim \alpha$.
\end{itemize}

\subsection{Continued fraction}\

Let $T: [0,1) \to [0,1)$ be the Gauss map, defined as
\begin{equation*}
 T(x)=\begin{dcases}
 	0,   &  x=0,\\
 	\frac{1}{x}- \bigg\lfloor  \frac{1}{x}\bigg\rfloor,   & x \in (0,1).
 \end{dcases}
\end{equation*}
For $x \in (0,1)$, take $a_1(x) = \lfloor 1 / x \rfloor$ and $a_{n+1} (x) = a_1 (T^n (x))$ for $ n \geq1$, which are called the partial~quotients of $x$.
Then $x$ can be written as its continued fraction expansion (see \eqref{CE}).
It is known that the continued fraction dynamical system $([0,1), T,  \mu) $ is ergodic (see \cite{Iosifescu02}), where $\mu$ is the Gauss measure given by
\[
d \mu: = \frac{1} {(1+x) \log 2 } dx.
\]
It is clear that $\mu$ is equivalent to the Lebesgue measure $\mathcal{L}$.

A finite truncation of the expansion of $x$ gives the rational $p_n/q_n = [a_1, a_2, \ldots, a_n]$, which is called the $n$th convergent of $x$.
With the conventions
\[
p_{-1} =1, \quad q_{-1}=0, \quad p_0 =1 \quad \text{and} \quad q_0=1,
\]
the sequences $p_n=p_n(x)$, $q_n=q_n(x)$ can be generated by the following recursive relation \cite{Khintchine63}:
\begin{equation}\label{pn}
  p_{n+1} = a_{n+1}(x) p_n + p_{n-1},  \quad q_{n+1} = a_{n+1}(x) q_n+q_{n-1}, \quad n \geq 0.
\end{equation}
Clearly, $p_n= p_n(x)$ and $q_n=q_n(x)$ are determined by the partial quotients $a_1, \ldots, a_n$,
so we may write
\[
p_n =p_n(a_1, \ldots, a_n), \quad q_n=q_n(a_1, \ldots, a_n).
\]
When it is clear which partial quotients are involved, we denote them by $p_n, q_n$ for simplicity.

For any $n \geq 1$, let $\mathbf{a}=(a_1,a_2,\ldots, a_n) \in \mathbb{N}^{n}$ be a positive integer vector, and we call
\[
I_k (\mathbf{a}) :=\{ x \in [0,1]: a_1(x)=a_1, a_2(x) =a_2, \ldots, a_n(x)=a_n\}
\]
a cylinder of order $n$ of continued fractions, i.e., the set of all real numbers in $[0,1]$ whose continued fraction expansions begin with $(a_1,a_2,\ldots, a_n) $.

We will frequently use the following well-known properties of continued fraction expansions.
These properties are discussed in  \cite{Iosifescu02, Khintchine63, Wu2006}.
\begin{proposition}\label{cylinder}
For any integer vector $\bm a=(a_1,a_2,\ldots, a_n)\in\mathbb N^n$, let $p_n=p_n(\bm a), q_n=q_n(\bm a)$ be defined in \eqref{pn}.
\begin{itemize}
    \item[(1)]
    \begin{equation*}
      I_n (\bm a) =\begin{dcases}
      	\bigg[\frac{p_n}{q_n}, \frac{p_n+p_{n-1}}{q_n+q_{n-1}} \bigg), &  \text{if }n\text{ is even},  \\
      	\bigg( \frac{p_n+p_{n-1}}{q_n+q_{n-1} }, \frac{p_n}{q_n} \bigg], &  \text{if }n\text{ is odd}.
      \end{dcases}
    \end{equation*}
  The length of $I_n (\bm a)$ is bounded by the inequalities
    \[
    \frac{1}{2 q_n^2} \leq |I_n (\bm a)|= \frac{1}{q_n (q_n+q_{n-1})} \leq \frac{1}{q_n^2}.
    \]
    \item[(2)]
    For $\bm b=(b_1,\dots,b_k)\in\mathbb N^k$, we have
    \[
    q_n(\bm a) \cdot q_k(\bm b)\le q_{n+k}(\bm a,\bm b) \le 2 q_n(\bm a) \cdot q_k(\bm b).
    \]
    \item[(3)]
    \[
    \frac{a_k+1}{2} \leq \frac{q_n(a_1, \cdots, a_n)}{q_{n-1}(a_1, \cdots,a_{k-1}, a_{k+1}, \cdots, a_n)} < a_k+1.
    \]
 \end{itemize}
\end{proposition}

The following result states that the Gauss map has an exponential mixing property with respect to the Gauss measure, which enables us to conclude a dynamical Borel-Cantelli lemma.
\begin{proposition}[\cite{Philipp67}]\label{MP}
There exist constants $C>0$ and $0< \gamma <1$ with the following property.
Fix $\mathbf{a}=(a_1,a_2,\ldots, a_k) \in \mathbb{N}^{k}$, and let $F \subset [0,1]$ be any measurable set.
Then for all $n \geq k+1$,
\[
\big|\mu(I_k (\mathbf{a}) \cap T^{-n} F) -\mu(I_k (\mathbf{a})) \mu(F) \big|  \leq C \mu(I_k (\mathbf{a})) \mu(F) \gamma^{\sqrt{n-k}}
\]
where $\mu$ is the Gauss measure.
\end{proposition}

To quantify the measure of a limsup set, the following lemmas are widely applied, which will also be used to show Theorem 1.4.

\begin{lemma}[{\cite[Borel--Cantelli Lemma]{Chung52}}]
	Let $(X, \mathcal{B}, \nu)$ be a finite measure space. Let $\{E_n\}_{n=1}^{\infty }$ be a countable family of measurable sets in $X$. Then
	\begin{equation}
		\nu \Big(\limsup_{n \to \infty} E_n \Big)  =0, \quad \text{if } \sum_{n=1}^{\infty} \nu(E_n) < \infty.
	\end{equation}
\end{lemma}

A result by Petrov \cite{Petrov02} provides a method for establishing lower bounds on the measure of certain $\limsup$ sets.
\begin{lemma}[{\cite[Theorem 2.1]{Petrov02}}]\label{l:quasiind}
	Let $(X, \mathcal{B}, \nu)$ be a probability space, and let $\{E_n\}_{n\ge 1}$ be a sequence of measurable subsets of $X$ such that
	\[\sum_{n=1}^{\infty}\mu(E_n)=\infty\]
	and
	\[\mu(E_i\cap E_j) \le c  \mu(E_i)\mu(E_j)\]
	for all $i, j>L$ with $i\ne j$, where $c\ge 1$ and $L$ are constants.
	Then,
	\[\mu\Big(\limsup_{n\to\infty}E_n\Big)\ge 1/c.\]
\end{lemma}

In light of Lemma \ref{l:quasiind}, the desired statement follows by showing that the sets $\Lambda_n (r, \psi(n))$, which are defined by
\[
\Lambda_n \big(r, \psi(n)\big)=\big\{ x \in [0,1): L_n(x, \psi(n)) \geq r \big\},
\]
are quasi-independent and that the sum of their measures diverges. It turns out that it is difficult to obtain a satisfactory estimate for the quasi-independence of the sets $\Lambda_n (r, \psi(n))$.
To address this, and motivated by \cite{Dolgopyat22}, we introduce the following notion for later use.
Let $A_{n}^{k}= \{x \in [0,1): a_k(x) \geq \psi(n)\}$.
Then
\[
\Lambda_n\big(r, \psi(n)\big) = \bigcup_{1\leq k_1 <k_2 <\cdots < k_r \leq n} \bigcap_{j=1}^{r} A_{n}^{k_j}.
\]
For an arbitrary $1=k_0 \leq k_1 <k_2<\cdots< k_r \leq n$, we consider the separation indices
\[
S_n(k_1, \ldots, k_r)= \#  \big\{0\le j\le r-1: k_{j+1} -k_j-1 \geq (\log n)^2\big\}.
\]
By the exponential mixing property of the Gauss measure (see Proposition \ref{MP}), we have the following results:

\begin{proposition}\label{Separated}
	(1) If $S_n(k_1, \ldots, k_r)=r$, then
	\begin{equation}\label{ieq:lu}
	\mu(A_{n}^{1}) ^r (1- C \gamma^{ \log n})^{r-1} \leq \mu\bigg(\bigcap_{j=1}^{r} A_{n}^{k_j} \bigg)\leq \mu(A_{n}^{1}) ^r (1+ C \gamma^{\log n})^{r-1}.
\end{equation}
where $C$ is as in Theorem \ref{MP}.

	(2)  If $S_n(k_1, \ldots, k_r) < r$, then there is a constant $K_r$ such that
	\[
\mu \bigg(\bigcap_{j=1}^{r} A_{n}^{k_j}\bigg) \leq K_r \mu(A_{n}^{1}) ^r .
	\]

	(3) If $1 \leq k_1 <k_2 < \cdots < k_r <l_1 < \cdots < l_r $ such that $2 ^{i-1} < k_\alpha \leq 2^i$
	and $2^{j-1} <l_\beta \leq 2^j$ for $1 \leq \alpha, \beta \leq r$, and
	\[
	S_{2^i}(k_1, \ldots, k_r) = S_{2^j}(l_1, \ldots, l_r)=r, \quad  l_1 -k_r \geq i^2,
	\]
	then there is a constant $c_3 >0$ such that
	\[
	\mu \bigg( \bigg[\bigcap_{\alpha =1} ^{r} A_{2^i}^{k_\alpha} \bigg] \cap \bigg[\bigcap_{\beta=1}^{r} A_{2^j}^{l_\beta}\bigg] \bigg)
	\leq \mu(A_{2^i}^{1} )^r  \mu(A_{2^j}^{1} )^r (1+c_3 \gamma^i).
	\]
\end{proposition}

\begin{proof}
	By Proposition \ref{MP}, for any $k\ge 1$ and any measurable set $E \subset [0,1]$, we have
	\begin{equation}\label{ieq:cylinder}
		\mu(I_k (\mathbf{a}) \cap T^{-n} E)  \leq  \mu(I_k (\mathbf{a})) \mu(E)  (1+C\gamma^{\sqrt{n-k}})
	\end{equation}
	and
	\begin{equation}
		\mu(I_k (\mathbf{a}) \cap T^{-n}E)  \geq  \mu(I_k (\mathbf{a})) \mu(E)  (1-C\gamma^{\sqrt{n-k}}),
	\end{equation}
	for all $\mathbf{a}=(a_1,a_2,\ldots, a_k) \in \mathbb{N}^{k}$ and $n \geq k+1$.

	(1) We show the inequality \eqref{ieq:lu} by induction on $r$.
	The inequality \eqref{ieq:lu} is obviously true for $r=1$,
	since $A_{n}^{k}= T^{-k+1} A_{n}^{1}$ and $\mu$ is invariant.

	Assume \eqref{ieq:lu} holds for $r=\ell$.
	Notice that for $r=\ell+1$,
	\begin{align*}
		\bigcap_{j=1}^{\ell+1} A_{n}^{k_j} &= \bigcap_{j=1}^{\ell} A_{n}^{k_j} \cap A_{n}^{k_{\ell+1}}\\
		& =\bigcup_{\substack{a_{k_i} \geq \psi(n)\\ 1\le i\le \ell}} I_{k_\ell}(a_1, \ldots, a_{k_\ell}) \cap T^{-k_{\ell+1}+1} A_{n}^{1}.
	\end{align*}
	Since the union is disjoint, by the countable additivity of the measure, one has
	\begin{align*}
		\mu \bigg(\bigcap_{j=1}^{\ell+1} A_{n}^{k_j} \bigg)
		&=\sum_{\substack{a_{k_i} \geq \psi(n)\\ 1\le i\le \ell}} \mu \bigg( I_{k_{\ell}}(a_1,  \ldots, a_{k_{\ell}}) \cap  T^{-(k_{\ell+1}+1)}A_{n}^{1} \bigg)\\
		&\leq \sum_{\substack{a_{k_i} \geq \psi(n)\\ 1\le i\le \ell} } \mu (I_{k_{\ell}}(a_1, \ldots, a_{k_{\ell}})) \mu (A_{n}^{1} ) (1+C \gamma^{\log n})\\
		&= \mu\bigg(\bigcap_{j=1}^{\ell} A_{n}^{k_j}\bigg) \mu ( A_{n}^{1}) (1+C\gamma^{\log n})\\
		&\leq  \mu(A_{n}^{1}) ^{\ell} (1+ C \gamma^{ \log n})^{\ell-1} \mu (A_{n}^{1}) (1+C\gamma^{\log n})\\
		&= \mu(A_{n}^{1}) ^{\ell+1} (1+ C\gamma^{ \log n})^{\ell},
	\end{align*}
	where the second inequality follows from \eqref{ieq:cylinder}, and the fourth inequality is obtained by the induction hypothesis.

	The opposite inequality follows from similar calculations, which completes the proof of item (1).

	(2)
	We present a detailed proof for the case $r=2$ only, since the general case  $r \geq 3$ follows a similar approach.

	By the invariance of $\mu$, it is obvious to see that
	\begin{align*}
		\mu\big( A_{n}^{k_1} \cap A_{n}^{k_2}\big) &=\mu\Big(T^{-k_1+1} \big(A_{n}^{1} \cap T^{-(k_2-k_1)} A_{n}^{1} \big)\Big)\\
		&=\mu\big(A_{n}^{1} \cap T^{-(k_2-k_1)}A_{n}^{1}\big) \\
		&=\mu\big(A_{n}^{1} \cap A_{n}^{k_2-k_1+1}\big).
	\end{align*}
	The intersection in the last equation can be written as a disjoint union of cylinder sets:
	\begin{align*}
		A_{n}^{1} \cap A_{n}^{k_2-k_1+1} &= \bigcup_{a \geq \psi(n)} \bigcup_{\bm u \in \mathbb{N}^{k_2-k_1-1}} \bigcup_{b \geq \psi(n)} I_{k_2-k_1+1}(a, \bm u,  b).
	\end{align*}
	Since $\mu$ is strongly equivalent to the Lebesgue measure with a density function $h(x)=\frac{1}{\log2} \cdot\frac{1}{1+x}$,
	\begin{align*}
		\mu( A_{n}^{k_1} \cap A_{n}^{k_2} )
		&=\sum_{a \geq \psi(n)} \sum_{\bm u \in \mathbb{N}^{k_2-k_1-1}} \sum_{b \geq \psi(n)}\mu( I_{k_2-k_1+1}(a,\bm u, b))  \\
		&\asymp \sum_{a \geq \psi(n)} \sum_{\bm u \in \mathbb{N}^{k_2-k_1-1}} \sum_{b \geq \psi(n)} \mathcal{L}( I_{k_2-k_1+1}(a,\bm u ,b)).
	\end{align*}
	By Proposition \ref{cylinder} (1), each $\mathcal{L}( I_{k_2-k_1+1}(a,\bm u, b)) $ in the summation can be estimated as
	\begin{align*}
		\mathcal{L}( I_{k_2-k_1+1}(a,\bm u ,b))
		&\asymp \frac{1}{q_{k_2-k_1+1}^{2}(a, \bm u, b)}\\
		&\asymp \frac{1}{q_1^{2}(a)q_{k_2-k_1-1}^{2}(\bm u) q_1^2(b)} \\
		&\asymp \mathcal{L}(I_1(a)) \mathcal{L}(I_{k_2-k_1-1}(\bm u)) \mathcal{L}( I_1(b)) \\
		& \asymp \mu(I_1(a)) \mu(I_{k_2-k_1-1}(\bm u)) \mu( I_1(b)).
	\end{align*}
	Therefore, we have
	\begin{align*}
		\mu( A_{n}^{k_1} \cap A_{n}^{k_2})
		&\asymp \sum_{a \geq \psi(n)} \mu(I_1(a)) \sum_{\bm u \in \mathbb{N}^{k_2-k_1-1}} \mu(I_{k_2-k_1-1}(\bm u)) \sum_{b \geq \psi(n)} \mu( I_1(b))\\
		&= \mu\bigg(\bigcup_{a \geq \psi(n)} I_1(a)\bigg)   \mu\bigg(\bigcup_{\bm u \in \mathbb{N}^{k_2-k_1-1} } I_{k_2-k_1-1}(\bm u )\bigg) \mu\bigg(\bigcup_{b \geq \psi(n)} I_1(b)\bigg) \\
		&=  \mu(A_n^1)    \mu(A_n^1),
	\end{align*}
which implies there is a constant $K_2$ such that
\[
\mu( A_{n}^{k_1} \cap A_{n}^{k_2}) \leq K_2 \mu(A_n^1)^2.
\]

	(3) In order to estimate the measure of $\big[\bigcap_{\alpha=1}^{r} A_{2^i}^{k_\alpha} \big] \cap \big[ \bigcap_{\beta=1}^{r} A_{2^j}^{l_\beta}\big] $, we represent the former set as a cylinder set, that is,
	\[
	\bigg[\bigcap_{\alpha =1} ^{r} A_{2^i}^{k_\alpha} \bigg] \cap \bigg[\bigcap_{\beta=1}^{r} A_{2^j}^{l_\beta}\bigg]=\bigcup_{\substack{a_{k_i} \geq \psi(n)\\ 1\le i\le r} } I_{k_r}(a_1, \ldots, a_{k_r} )  \cap T^{-l_1}
	\bigg(\bigcap_{\beta=1}^{r} A_{2^j}^{l_\beta-l_1} \bigg)
	\]
	By \eqref{ieq:cylinder} and the condition $l_1-k_r\ge i^2$, the measure of each set on the right-hand side is majorized by
	\begin{align*}
		\leq \mu( I_{k_r}(a_1, \ldots, a_{k_r} )) \mu \bigg( \bigcap_{\beta=1}^{r} A_{2^j}^{l_\beta}\bigg) (1+C \gamma^i),
	\end{align*}
	where we have used $\mu \big( \big[\bigcap_{\beta=1}^{r} A_{2^j}^{l_\beta-l_1} \big] \big)=\mu \big( \big[\bigcap_{\beta=1}^{r} A_{2^j}^{l_\beta} \big] \big)$.
	Summing over $a_{k_1}, \ldots, a_{k_r}$ and noting that it is disjoint union,
	it yields
	\begin{align*}
		&\mu \bigg( \bigg[\bigcap_{\alpha =1} ^{r} A_{2^i}^{k_\alpha} \bigg] \cap \bigg[\bigcap_{\beta=1}^{r} A_{2^j}^{l_\beta}\bigg] \bigg)
		\leq \mu \bigg(\bigcap_{\alpha=1}^{r} A_{2^i}^{k_\alpha} \bigg) \mu\bigg(\bigcap_{\beta=1}^{r} A_{2^j}^{l_\beta}\bigg)  \big(1+C \gamma^i \big).
	\end{align*}
	Combining the estimates for the measures of the sets $\bigcap_{\alpha=1}^{r} A_{2^i}^{k_\alpha}$ and $\bigcap_{\beta=1}^{r} A_{2^j}^{l_\beta}$ presented in item (1),
	one has
	\begin{align*}
	\mu\bigg( \bigg[\bigcap_{\alpha =1} ^{r} A_{2^i}^{k_\alpha} \bigg] \cap \bigg[\bigcap_{\beta=1}^{r} A_{2^j}^{l_\beta}\bigg] \bigg)
	\leq  &\big(\mu(A_{2^i}^{1})\big)^r \big(\mu(A_{2^j}^{1})\big )^r  (1+ C\cdot \gamma^{i \log 2})^{r-1} \\
	 &\cdot (1+ C\cdot \gamma^{j \log 2})^{r-1} (1+C\gamma^i)\\
	 \lesssim & ~ \mu(A_{2^i}^{1} )^r  \mu(A_{2^j}^{1} )^r (1+c_3 \gamma^i).
	  \qedhere
	\end{align*}
\end{proof}

\subsection{Hausdorff dimension}\

For completeness we give a brief introduction to Hausdorff measure and dimension. For further details we refer to the beautiful text \cite{Falconer2013}.

Let $ (X,d) $ be a metric space and $\mathcal{D}(Z,r)$ denote the collection of countable open covers $\{U_{i}\}_{i=1}^{\infty}$ of $Z$ for which ${\rm diam}\, U_{i} < r$ for all $i$. Given a subset $ Z \subset X, t\geq 0$, define
\[
\mathcal{H}^t(Z, r)=\inf_{\mathcal{D}(Z,r)}\left\{\sum\limits_{U_{i} \in \mathcal{D}(Z,r)} ({\rm diam}\,U_{i})^t \right\}.
\]
It is easy to see that $\mathcal{H}^t(Z, r)$ increases as $r$ decreases. Therefore there exists the limit
\[
\mathcal{H}^t(Z)=\lim\limits_{r\rightarrow 0} \mathcal{H}^t(Z, r)
\]
which is called the \emph {t-dimensional Hausdorff measure} of $Z$.
The number
\[
\dim_{\mathrm{H}}Z=\inf\{t >0: \mathcal{H}^t(Z)=0\}=\sup\{t>0 : \mathcal{H}^t(Z)=\infty\}
\]
is called the \emph {Hausdorff dimension} of $Z$.

\subsection{Topological pressure}\

When dealing with problems related to dimension theory, concepts such as the pressure function from thermodynamics prove to be powerful tools. The notion of a general pressure function was introduced by Ruelle \cite{Ruelle1973} and Walters \cite{Walters1975, Walters1982}, describing  the exponential growth rate of ergodic sums.
Our focus lies in methods for computing Hausdorff dimensions via pressure functions.
For more details and context on pressure functions in infinite conformal iterated function systems can be found in \cite{MU1996, MU1999, MU2003}.

Given any real-valued function $\varphi : [0,1) \to \mathbb{R}$, the pressure function is defined as
\begin{equation}\label{pressure}
	P(T, \varphi):= \lim_{n \to \infty} \frac{1}{n} \log \sum_{a_1,\ldots, a_n \in \mathbb{N}} \sup_{x \in [0,1)} e^{S_n \varphi ([a_1, \ldots, a_n+x])},
\end{equation}
where $S_n \varphi (x)$ denotes the $n$th ergodic sum $\varphi (x)+ \cdots+ \varphi (T^{n-1} x)$ at $x$.

For any $n \in \mathbb{N}$, the notation
\[
Var_n(\varphi): =\sup \{| \varphi(x)- \varphi(y) | : I_n(x) = I_n(y) \}
\]
is called the $n$th variation of $\varphi$, where $I_n(x)$ represents the $n$th order cylinder containing $x \in [0, 1)$.

The existence of the limit in equation (\ref{pressure}) is guaranteed by the following proposition given by Li et al \cite{Li2014}.

\begin{proposition}[\cite{Li2014}]\label{p:pressure}
	Let $\varphi : [0, 1) \to \mathbb{R}$ be a real function with $Var_1 (\varphi) < \infty$ and $Var_n(\varphi) \to 0$ as $n \to \infty$.
	Then the limit defining $P(T, \varphi) $ exists and the value of $P(T, \varphi)$ remains the same even without taking supremum over $x \in [0,1)$ in (\ref{pressure}).
\end{proposition}
Note that if $\varphi$ is continuous, then $\varphi$ satisfies the variation condition and the supremum in equation \eqref{pressure} can be removed.

For any $1 < B < \infty$ and $s \geq 0$, it is clear that the $n$th variation of the potential
\[
- s \log |T'|-\big(s+(2s-1)(r-1)\big) \log B
\]
goes to zero as $n\to\infty$. This, together with Proposition \ref{p:pressure}, allows us define the dimensional number
\[
s(r, B):=	\inf \Big\{ s \geq 0: P\Big(T, - s \log |T'|-\big(s+(2s-1)(r-1)\big) \log B\Big) \leq 0 \Big\}.
\]
With little effort, the pressure function in the definition of $s(r,B)$ can be written more explicitly as
\[P\Big(T, - s \log |T'|-\big(s+(2s-1)(r-1)\big) \log B\Big)=  \lim_{n \to \infty} \frac{1}{n} \log \sum_{\bm a\in \mathbb{N}^n}   \frac{1}{ q_n^{2s}(\bm a)B^{n(s+(2s-1)(r-1))}}.\]
%
%
Following a similar argument as in \cite{Wang2008} with some minor modifications, we can derive the following from the right-hand side of the above equation:
\begin{proposition}[\cite{Wang2008}] \label{p:s(B)}
Let $r\ge 1$ be an integer. It holds that
\begin{itemize}
		\item[(1)] As a function of $B \in (1, \infty)$, $s(r, B)$ is strictly decreasing and continuous;
		\item[(2)]
		\[
		\lim_{B \to 1} s(r, B) =1 \quad \text{and} \quad \lim_{B \to \infty} s(r, B)=1/2.
		\]
\end{itemize}

\end{proposition}

\section{ \emph{Proof of Theorem \ref{leb}}}
Recall that
\[
F(r, \psi): =\Big\{ x \in [0,1): L_n(x, \psi(n)) \geq r\  \text{for i.m.}~n \in \mathbb{N} \Big\}.
\]
We start with a proposition that allows us to replace $\psi$ by a non-decreasing function, and the resulting set remains unchanged.
\begin{proposition}\label{p:mon}
Let $\widetilde{\psi}(n) = \min_{m\geq n} \psi(m)$. We have $F(r, \psi)=F(r, \widetilde{\psi})$.
\end{proposition}

\begin{proof}
By the definition of $\widetilde{\psi}$, we have $\widetilde{\psi}(n)  \leq \psi(n)$.
Hence, $F(r, \psi) \subset F(r, \widetilde{\psi})$.

On the other hand, for any $x \in F(r, \widetilde{\psi})$, there exist infinitely many $n$ such that
\begin{equation}\label{eq:akj>}
	a_{k_j}\ge \widetilde{\psi}(n) =\min_{m \geq n} \psi(m) \quad \text{for some $k_j\le n$ ($1\le j\le r$)}.
\end{equation}
For any such $n$, there exists an integer $m_n\ge n$ such that $\widetilde{\psi}(n) =\psi(m_n)$. In particular, it follows from \eqref{eq:akj>} that
\begin{equation}\label{eq:akj>2}
	a_{k_j}\ge \psi(m_n)=\widetilde{\psi}(n)\quad \text{for some $k_j\le n\le m_n$ ($1\le j\le r$)}.
\end{equation}
This means that there will also exist infinitely many $m_n$ such that \eqref{eq:akj>2} holds, and thus $ x \in F(r, \psi) $.
This yields $F(r, \widetilde{\psi})  \subset F(r, \psi)$, as desired.
\end{proof}

In view of Proposition \ref{p:mon}, from now on we will assume that $\psi$ is non-decreasing, i.e.
\begin{equation}\label{eq:psi}
	\min_{m\geq n}\psi(m)=\psi(n).
\end{equation}
Therefore,
\[\sum_{n=1}^{\infty} n^{r-1} \Big(\min_{m\geq n}\psi(m)\Big)^{-r}=\sum_{n=1}^{\infty} n^{r-1} \psi(n)^{-r}.\]

\subsection{Convergence part} \
In this subsection, assume that
\begin{equation}\label{eq:conver}
	\sum_{n=1}^{\infty} n^{r-1} \psi(n)^{-r} < \infty.
\end{equation}
Recall that
\[
\Lambda_n \big(r, \psi(n)\big)=\big\{ x \in [0,1): L_n(x, \psi(n)) \geq r \big\}.
\]
The limsup nature of the set $F(r, \psi) $ enables us to rewrite it as follows:
\begin{align}\label{cover}
	F(r, \psi)=&\bigcap_{N=1}^\infty\bigcup_{n=N}^\infty \Lambda_n\big(r, \psi(n)\big) \notag\\
	= &\bigcap_{N=1}^\infty\bigcup_{m=N}^\infty\bigcup_{n=2^m} ^{2^{m+1}}  \Lambda_n\big(r, \psi(n)\big)  \notag \\
	\subset & \bigcap_{N=1}^\infty\bigcup_{m=N}^\infty \Lambda_{2^{m+1}} \big(r, \psi(2^m)\big),
\end{align}
where the inclusion holds by the monotonicity of $\psi$.

We claim that the convergence assumption described in \eqref{eq:conver} is equivalent to
\begin{equation}\label{eq:convergence}
	\sum_{m=1}^\infty \mu \Big(\Lambda_{2^{m+1} } \big(r, \psi(2^m) \big) \Big) < \infty,
\end{equation}
which combined with Borel--Cantelli Lemma yields,
\[
\mu \bigg(\bigcap_{N=1}^\infty \bigcup_{m = N}^\infty  \Lambda_{2^{m+1}} \big(r, \psi(2^m)\big) \bigg) =0.
\]
By (\ref{cover}), it follows
\[
\mu(F(r, \psi))=0,
\]
which completes the proof of the convergence part of Theorem \ref{leb}.

Now, we prove the claim (see \eqref{eq:convergence}).
By Proposition \ref{Separated} (2), we obtain
\begin{align*}
	\mu\Big(\Lambda_{2^{m+1} } \big(r, \psi(2^m)\big)\Big)  &
	\leq \sum_{1 \leq k_1 < \cdots < k_r
	\leq 2^{m+1}} \mu\bigg(\bigcap_{j=1} ^r A_{2^m}^{k_j} \bigg)\\
	& \leq \sum_{1 \leq k_1 < \cdots < k_r
		\leq 2^{m+1}} K_r\mu (A_{2^m} ^1)^r\\
	&\lesssim  2^{rm} \mu (A_{2^m} ^1)^r,
\end{align*}
where the implied constant depends on $r$ only.
Since $\mu$ is strongly equivalent to the Lebesgue measure with a density function $h(x)=\dfrac{1}{\log2} \dfrac{1}{1+x}$, it follows that $\mu(A_{2^m} ^1) \asymp \mathcal{L} (A_{2^m} ^1)= \psi(2^m)^{-1} $.
Hence,
\begin{equation}\label{eq:estimate}
	\mu \big(\Lambda_{2^{m+1} } \big(r, \psi(2^m)\big) \big)  \lesssim (2^m  \psi(2^m)^{-1})^r.
\end{equation}
Note that
\begin{equation}\label{eq:bound}
\sum_{n=2^j}^{2^{j+1}-1} n^{r-1} \psi (n)^{-r}
\leq \big(2^{j+1} \psi(2^j)^{-1} \big)^{r}
\leq 2^{2r} \sum_{n=2^{j-1}}^{2^j-1} n^{r-1} \psi (n)^{-r},
\end{equation}
since $\psi$ is non-decreasing.
Hence, the convergence of $ \sum_{j=1}^{\infty} (2^j \psi(2^j)^{-1})^{r}$ is equivalent to that of $\sum_{n=1}^{\infty} n^{r-1} \psi(n)^{-r}$, and the proof of the claim is completed.

\subsection{Divergence part} \
In this subsection, assume that
\[
\sum_{n=1}^{\infty} n^{r-1} \psi(n)^{-r} = \infty.
\]

In order to estimate  the lower bound of $\mu(F(r, \psi))$, we construct a subset of  $F(r, \psi)$ by replacing $\psi$ with another function $\phi$, as follows:

\begin{lemma}\label{NF}
There exists a non-decreasing function $\phi \geq \psi$ such that $n \phi(n)^{-1} \to 0$ as $n \to \infty$ and
\[
\sum_{n=1}^{\infty} n^{r-1} \phi(n)^{-r} = \infty.
\]
\end{lemma}

\begin{proof}
Since $\psi$ is non-decreasing (see \eqref{eq:psi}), by the discussion presented in the last subsection (see \eqref{eq:bound}), one has
\begin{equation}
\sum_{n=1} ^{\infty} n^{r-1} \psi(n)^{-r} = \infty  \quad \Longleftrightarrow\quad \sum_{j=1}^{\infty} 2^{jr} \psi(2^j)^{-r} =\infty.
\end{equation}
Thus, we can choose a sequence of integers $(m_k)_{k\geq 1}$ with $m_1=1$ such that
\[
\sum_{j=m_k }^{m_{k+1}-1} 2^{jr} \psi(2^j)^{-r} \geq \frac{1}{k}.
\]
For any $ k \in \mathbb{N}$ and $n \in [2^{m_k}, 2^{m_{k+1}-1}] \cap \mathbb{N}$, define
\[
\phi(n) = \max \Big\{ \psi (n), nk^{1/r} \Big\}.
\]

It remains to verify $\phi$ satisfies the desired property.
Since both functions on the right-hand side are non-decreasing, their maximum $\phi$ is also non-decreasing.
Moreover, it is easy to see that $\phi(n) \geq \psi(n)$ for $n \in  \mathbb{N}$.
By the definition of $\phi$, it follows that
\[
n\phi(n)^{-1} = \min\Big\{ n\psi (n)^{-1},k^{-1/r}\Big\} \leq k^{-1/r}\to 0 \quad ( n \to \infty).
\]
Again, the monotonicity of $\phi$ implies that
\begin{equation}
\sum_{n=1}^{\infty} n^{r-1} \phi(n)^{-r} = \infty\quad\Longleftrightarrow\quad\sum_{j=1}^{\infty} 2^{jr} \phi(2^j)^{-r} =\infty.
\end{equation}
Thus, the conclusion of the lemma will follow if the series on the right-hand side diverges. Note that
\begin{align*}
\sum_{j=1}^{\infty} 2^{jr} \phi(2^j)^{-r}=\sum_{k=1}^{\infty} \sum_{j=m_k}^{m_{k+1}-1} 2^{jr} \phi(2^{j})^{-r}.
\end{align*}
For any $k\ge 1$, by the definition of $\phi$, there are two cases to discuss.

Case 1: There exists $j\in [m_k,m_{k+1}-1]$ such that $\phi(2^j)= 2^j k^{1/r}$. Then,
\[
 \sum_{j=m_k}^{m_{k+1}-1} 2^{jr} \phi(2^{j})^{-r} \geq\frac{1}{k}.
\]

Case 2: For any $j\in [m_k,m_{k+1}-1]$, we have $\phi(2^j)= \psi(2^j)$. By our choice of the sequence of $\{m_k \}_{k\ge 0}$,
\[
 \sum_{j=m_k}^{m_{k+1}-1} 2^{jr} \phi(2^{j})^{-r}= \sum_{j=m_k}^{m_{k+1}-1} 2^{jr} \psi(2^{j})^{-r}\ge\frac{1}{k}.
\]
Combining the estimates given in Cases 1--2, we deduce that
\[\sum_{k=1}^{\infty} \sum_{j=m_k}^{m_{k+1}-1} 2^{jr} \phi(2^{j})^{-r}\ge \sum_{k=1}^{\infty}\frac{1}{k}=\infty,\]
 which completes the proof.
\end{proof}

By Proposition \ref{p:mon} and Lemma \ref{NF}, it is safe to assume that
\begin{equation}
	\psi~\text{is non-decreasing}\quad\text{and}\quad n \psi(n) ^{-1} \to 0 \text{ as } n \to \infty.
\end{equation}
Moreover, we still has
\[\sum_{n=1}^{\infty} n^{r-1} \psi(n)^{-r} = \infty.\]

For $m \in \mathbb{N}$ with $m \geq 10$, let
\[
\mathcal{U}_m := \left\{ (k_1, \ldots, k_r) : 2^{m-1} +m^{2} < k_1 < \cdots < k_r \leq 2^{m} \text{ and}\  S_{2^m} (k_1, k_2, \ldots, k_r)=r\right\}.
\]
\begin{remark}
	To ensure that the sets under consideration are quasi-independent (see Lemma \ref{l: independent}), here we require $k_1>2^{m-1}+m^2$ instead of $k_1>2^{m-1}$.
\end{remark}
Recall that
$A_{n}^{k}= \{x \in [0,1): a_k(x) \geq \psi(n)\}$
and
\[
\Lambda_n \big(r, \psi(n)\big) = \bigcup_{1\leq k_1 <k_2 <\cdots < k_r \leq n} \bigcap_{j=1}^{r} A_{n}^{k_j}.
\]
Define
\[
F_m :=  \bigcup_{ (k_1, \ldots, k_r) \in \mathcal{U}_m } \bigcap_{j=1}^{r} A_{2^m}^{k_j}.
\]
Clearly, $F_m \subset \Lambda_{2^m} (r, \psi (2^m))$.
By the monotonicity of $\psi$,
\begin{equation}\label{eq:cup}
	\bigcap_{N=1}^\infty \bigcup_{m=N}^\infty F_m \subset \bigcap_{N=1}^\infty \bigcup_{m=N}^\infty \Lambda_{2^m} \big(r, \psi (2^m)\big) \subset F(r, \psi).
\end{equation}

We aim to show that $\bigcap_{N=1}^\infty \bigcup_{m=N}^\infty F_m$ has full measure
if the series $\sum_{n=1}^{\infty} n^{r-1} \psi(n)^{-r}$ diverges.

\begin{lemma}\label{ME}
	We have
\[
\mu(F_m) = \frac{2^{r(m-1)}} {r !} \big(1+ o(1)\big) \mu(A_{2^m}^1)^r,
\]
where the error term $o(1)$ tends to $0$ as $m\to\infty$.
\end{lemma}

\begin{proof}
We first prove that the set $\mathcal{U}_m$ contains approximately $\frac{2^{r(m-1)}}{r!}\big(1+o(1)\big)$ elements, where the term $o(1)$ tends to zero as $m \to \infty$.
Specifically, we will select $r$ positions from the interval $[2^{m-1} + m^2,\, 2^m]$ for arrangement.

For the first selection, the number of available choices is
\[
2^m - (2^{m-1} + m^2) = 2^{m-1} - m^2 = 2^{m-1}\big(1+o(1)\big),
\]
Similarly, for the second selection, taking into account the separation condition $k_{j+1} - k_j \geq (m \log 2)^2$, the number of available choices becomes
\[
2^{m-1} - m^2 - 1 - 2m^2 = 2^{m-1} - 3m^2 - 1 = 2^{m-1}\big(1+o(1)\big).
\]
Proceeding inductively, each subsequent choice also admits $2^{m-1}\big(1+o(1)\big)$ possibilities.
Thus, the total number of arrangements satisfies
\begin{equation}\label{eq:cardinum}
	\# \mathcal{U}_m = \frac{\big(2^{m-1}\big(1+o(1)\big)\big)^r}{r!}
	= \frac{2^{r(m-1)} \big(1+o(1)\big)}{r!}.
\end{equation}
Therefore, by the definition of $F_m$ and  Proposition \ref{Separated} (1),
\begin{equation}\label{eq:estmiateIm}
	\mu(F_m) \leq \sum_{(k_1, \ldots, k_r) \in \mathcal{U}_m} \mu\bigg(\bigcap_{j=1}^r A_{2^m}^{k_j} \bigg)=\frac{2^{r(m-1)}} {r !} \big(1+ o(1)\big) \mu(A_{2^m}^1)^r.
\end{equation}

The lower bound estimate for $\mu(F_m)$ is more delicate, as the sets involved in the last inequality are not necessarily disjoint. By an elementary property in measure theory,
\begin{align*}
  \mu(F_m) &  \geq \sum_{(k_1, \ldots, k_r) \in \mathcal{U}_m} \mu \bigg(\bigcap_{j=1}^r A_{2^m}^{k_j} \bigg) -
 \sum_{\substack{ (k_1, \ldots, k_r) \in \mathcal{U}_m\\
  (k_1 ', \ldots, k_r ') \in \mathcal{U}_m \\
  (k_1, \ldots, k_r) \neq (k_1 ', \ldots, k_r ')} }   \mu \bigg( \bigg[\bigcap_{j =1} ^{r} A_{2^m}^{k_j}\bigg] \cap \bigg[\bigcap_{j=1}^{r} A_{2^m}^{k_j '} \bigg] \bigg) \\
  &:= I_m -J_m .
\end{align*}
The estimate for $I_m$ is already provided in \eqref{eq:estmiateIm}, and we will now focus on estimating $J_m$. Observe that for any $(k_1, \ldots, k_r) \neq (k_1 ', \ldots, k_r ')$, the intersection
\[
\bigg[\bigcap_{j =1} ^{r} A_{2^m}^{k_j}\bigg] \cap \bigg[\bigcap_{j=1}^{r} A_{2^m}^{k_j '}\bigg] = \bigcap_{i \in \{k_1, \ldots, k_r\} \cup \{ k_1', \ldots, k_r'\}} A_{2^m}^{i}
\]
involves at least $r+1$ sets. It then follows that
\begin{align*}
	J_m & \leq \sum_{k_1 < \cdots < k_{r+1}} \mu \bigg(\bigcap_{j=1}^{r+1} A_{2^m}^{k_j}\bigg)  \\
	& \leq C_{2^m}^{r+1}\cdot K_{r+1} \mu( A_{2^m}^1)^{r+1}\\
	& \lesssim  \frac{2^{r(m-1)}}{r!} 2^{m}\cdot K_{r+1} \mu( A_{2^m}^1)^{r}\cdot  \psi(2^m)^{-1} \\
	&= o(1) \frac{2^{r(m-1)}}{r!} \mu( A_{2^m}^1)^r,
\end{align*}
where Proposition~\ref{Separated}~(2) is used in deriving the second inequality, and the fact that $n \psi^{-1}(n) \to 0$ as $n \to \infty$ is applied in the final equality.

Combining the estimate of $I_m$ and $J_m$, the conclusion follows.
\end{proof}

Now, we derive the following quasi-independent result.

\begin{lemma}\label{l: independent}
For any $m < n$ with $m \geq 10$, one has
\[
\mu(F_m \cap F_n ) \leq  \mu(F_m) \mu(F_n) \big(1 + o(1)\big),
\]
where the error term $o(1)$ tends to $0$ as $m\to\infty$.
\end{lemma}

\begin{proof}
	By the definition of $E_m$, one has
	\[
	F_m \cap F_n =\bigg[\bigcup_{ (k_1, \ldots, k_r) \in \mathcal{U}_m } \bigcap_{j=1}^{r} A_{2^m}^{k_j} \bigg] \cap \bigg[\bigcup_{(l_1, \ldots, l_r) \in \mathcal{U}_n } \bigcap_{j=1}^{r} A_{2^n}^{l_j}\bigg].
	\]
	Since $m<n$, by the definitions of $\mathcal{U}_m$ and $\mathcal{U}_n$, we have
	\[l_1>2^{n-1}+n^2>2^m+m^2>k_r+m^2.\]
	This means that Proposition \ref{Separated} (3) is applicable. By that conclusion,
\begin{align*}
  &\mu(F_m \cap F_n) \\
  \leq &  \sum_{(k_1, \ldots, k_r) \in \mathcal{U}_m} \sum_{ (l_1, \ldots, l_r) \in \mathcal{U}_n } \mu\bigg( \bigg[\bigcap_{j=1}^{r} A_{2^m}^{k_j}\bigg] \cap \bigg[\bigcap_{j=1}^{r} A_{2^n}^{l_j}\bigg]\bigg)\\
    \leq & \#\mathcal{U}_m\cdot \#\mathcal{U}_n\cdot \mu \big(A_{2^m}^1\big)^r \mu \big(A_{2^n}^1\big)^r(1+ c_3 \gamma^m)\\
   =& \big(1+o(1) \big) \mu(F_m) \mu(F_n),
\end{align*}
where we use $\# \mathcal{U}_m =\dfrac{2^{r(m-1)} \big(1 + o(1) \big)}{r!}$ (see \eqref{eq:cardinum}) and Lemma \ref{ME} in the last equality.
\end{proof}

\begin{proof}[Proof of Theorem \ref{leb}: divergence part]
For any $\epsilon>0$, by Lemma \ref{l: independent}, there exists $M=M(\epsilon)>0$ such that for any $n>m>M$,
\[\mu(F_m \cap F_n ) \leq  \mu(F_m) \mu(F_n)(1 + \epsilon).\]
Since the series $\sum_{n=1}^{\infty} n^{r-1} \psi(n)^{-r}$ diverges, by Lemma \ref{l:quasiind} and \eqref{eq:cup},
\[\mu\big(F(r, \psi)\big)\geq \mu \Big(\limsup_{m \to \infty} F_m\Big)\ge 1/(1 + \epsilon).\]
Letting $\epsilon\to 0$, we conclude
\[\mu\big(F(r, \psi)\big)\ge 1,\]
which completes the proof.
\end{proof}

\section{ \emph{Proof of Theorem \ref{Hausdorff}}}

The computation of the Hausdorff dimension of a set typically involves two steps: establishing upper and lower bounds separately.
The upper bound is usually obtained by constructing a suitable covering; in particular, coverings of limsup sets often arise naturally in this context.
To obtain the lower bound, one generally constructs an appropriate subset.
However, in our framework, we are fortunate to be able to apply known results due to Hussain and Shulga \cite{Hussain23},  {\L}uczak \cite{Luczak1997}, and Feng, Wu, Liang, and Tseng \cite{Feng1997}  (see Theorems \ref{t:dimension} and \ref{thm:dim}).

In this section, we may continue to that $\psi(n)$ is non-decreasing, as guaranteed by Proposition \ref{p:mon}.
Under this assumption, the quantities $\log B$ and $\log b$ can be rewritten as
\[
\log B: = \liminf_{n \to \infty} \frac{\log \psi(n)}{n} \quad \text{and} \quad \log b: = \liminf_{n \to \infty} \frac{\log \log \psi(n)}{n}.
\]

Recall that
\[
s(r, B)=\inf \Big\{ s \geq 0: P\Big(T, - s \log |T'|-\big(s+(2s-1)(r-1)\big) \log B\Big) \leq 0 \Big\}.
\]

In what follows, we compute the Hausdorff dimension of the set $F(r, \psi)$ by analyzing three cases according to the value of $B$.

\subsection{The Case $1< B < \infty$}\

We begin by estimating the upper bound of the Hausdorff dimension of the set $F(r, \psi)$.

For any $n\in\mathbb N$ and $a\ge 1$ , define
\begin{align*}
\Gamma_n(a) = \big\{ x \in [0,1): a_n(x) \geq  a \text{ and } &\exists~ 1 \leq k_1 < \cdots <k_{r-1} < n\\
&~\text{s.t.}~a_{k_j}(x) \geq a \text{ for }1 \leq j \leq r-1 \big\}.
\end{align*}

\begin{lemma}\label{l:cover}
We have
	\[
	F(r, \psi) =\bigcap_{N = 1}^{\infty} \bigcup_{n =N}^{\infty} \Gamma_n(\psi(n)).
	\]
\end{lemma}

\begin{proof}
Clearly, the set on the right-hand side is contained in $F(r, \psi)$
 by the definition of $ F(r, \psi)$.

Conversely, let $x \in F(r, \psi)$ be arbitrary.
Then there exists an infinite sequence $\{n_i\}_{i \geq 1} \subset \mathbb{N}$
such that for each $i \geq 1$, there exist indices $ k_1^i, \ldots, k_r^i $
satisfying
\[
1 \leq  k_1^i < \cdots < k_r^i \leq n_i ,
\]
with
\[
a _{k_j^i}(x) \geq \psi(n_i) \quad \text{for~}  1 \leq  j \leq r.
\]
Since $\psi(n) \to \infty$ as $n \to \infty$,
a fixed index $k_r^{i_0}$ cannot correspond to infinitely many $n_i$.
Therefore, the sequence $\{k_r^i\}_{i \geq 1}$ must contains infinitely many distinct terms.
For each $k_j^i$, by the monotonicity of $\psi$ and the previous inequality, we have
	\[
	a _{k_j^i}(x) \geq \psi(n_i) \geq \psi(k_r^i)\quad \text{for}~1\leq j \leq r-1,
	\]
	and
	\[
	a _{k_r^i}(x) \geq \psi(k_r^i).
	\]
Hence, $x$ belongs to the set on the right-hand side.
This completes the proof.
\end{proof}

By the definition of $B$, for any $\varepsilon>0$, there exists $N_0 \in \mathbb{N}$ such that for all $ n \geq N_0$, we have $\psi(n) \geq (B-\varepsilon)^n$.
Therefore,
\begin{equation}\label{eq:contain}
F(r, \psi) \subset  \big\{ x \in [0, 1): L_n (x, (B-\varepsilon)^n) \geq r ~\text{for i.m.}~n \big\},
\end{equation}
where
\[
L_n (x, (B-\varepsilon)^n): = \# \{k \leq n : a_k(x) \geq (B-\varepsilon)^n\}.
\]
For convenience, define the set
\[
F(r, B): = \{ x \in [0, 1): L_n (x, B^n) \geq r ~\text{for i.m.}~n \}.
\]
Then by \eqref{eq:contain}, we have $ F(r, \psi) \subset F(r, B-\varepsilon) $, and thus
\begin{equation}\label{eq:ub}
\dim_{\mathrm{H}} F(r, \psi) \leq \dim_{\mathrm{H}} F(r, B-\varepsilon).
\end{equation}

In view of \eqref{eq:ub}, our next goal is to estimate the upper bound of the Hausdorff dimension of $F(r, B)$ for $1<B<\infty$. Applying Lemma \ref{l:cover} with  $\psi(n)=B^n$ yields
\begin{equation}\label{eq:limsup}
	F(r, B) =\bigcap_{N =1}^\infty \bigcup_{n=N}^\infty \Gamma_n(B^n).
\end{equation}
By the definition of $\Gamma_n(B^n)$,
\begin{align}
\Gamma_n(B^n)&=\bigcup_{1 \leq  k_1 < \cdots < k_{r-1}<n} \bigcup_{ \substack{ \bm a \in \mathbb{N}^{n-1}\\
		a_{k_j} \geq B^n, 1 \leq j \leq r-1}}  \bigcup_{a_n \geq B^n} I_{n}(\bm a, a_n),\notag\\
&:=\bigcup_{1 \leq  k_1 < \cdots < k_{r-1}<n} \bigcup_{ \substack{ \bm a \in \mathbb{N}^{n-1}\\
		a_{k_j} \geq B^n, 1 \leq j \leq r-1}} J_{n-1}(\bm a).\label{eq:gamman}
\end{align}

We are now ready to establish the upper bound for the Hausdorff dimension of the set $F(r,\psi)$.

\begin{proof}[Proof of upper bound]
	By \eqref{eq:ub}, it suffices to verify the upper bound for the Hausdorff dimension of $F(r, B)$.
	Let $t > s(r, B)$.
	By the definition of $s(r,B)$,
	\[
	P(T, -t \log |T'|-(t+(2t-1)(r-1) )\log B) =\lim_{n \to \infty} \frac{1}{n} \log \sum_{\bm a \in \mathbb{N}^n}   \frac{1}{B^{n(t+(2t-1)(r-1))}q_n^{2t}(\bm a)} < 0.
	\]
	This implies that there exists $\delta>0$ such that the sum
	\begin{equation}\label{eq:sum}
		\sum_{\bm a \in \mathbb{N}^n}   \frac{1}{B^{n(t+(2t-1)(r-1))}q_n^{2t}(\bm a)} < e^{-n\delta} \quad \text{whenever $n$ is large enough}.
	\end{equation}

	Note that for any $N\ge 1$, $\bigcup_{n=N}^\infty \Gamma_n(B^n)$ forms a cover of $F(r,B)$.
	To estimate upper bound of the Hausdorff dimension of $F(r, B)$, it is necessary to find an optimal cover for each $\Gamma_n(B^n)$. For this purpose, using the expression of $\Gamma_n(B^n)$ given in \eqref{eq:gamman}, we first calculate the length of $J_{n-1}(\bm a)$, namely
	\begin{align} \label{eq:length}
		|J_{n-1}(\bm a)| &\leq \sum_{a_n \geq B^n} \frac{1}{q_{n}^2 (\bm a, a_{n})}\leq \sum_{a_n \geq B^n} \frac{1}{q_{n-1}^2 (\bm a) a_n^2}  \notag \\
		& \asymp \frac{1}{q_{n-1}^2 (\bm a) B^n}.
	\end{align}
It then follows that the $t$-volume cover of $\Gamma_n(B^n)$ is bounded from above by
\begin{equation}\label{eq:tcoverofgamma}
	\sum_{1 \leq k_1 < \cdots < k_{r-1} < n} \sum_{ \substack{ \bm a \in \mathbb{N}^{n-1}\\
			a_{k_j} \geq B^n, 1 \leq j \leq r-1}}   \frac{1}{q_{n-1}^{2t} (\bm a) B^{nt}}.
\end{equation}
Fix integers $k_1, k_2, \ldots, k_{r-1}$ with $1< k_1<\cdots < k_{r-1}<n$, and consider the inner sum
\[
\sum_{\substack{ \bm a \in \mathbb{N}^{n-1}\\
		a_{k_j} \geq B^n, 1 \leq j \leq r-1}}   \frac{1}{q_{n-1}^{2t} (\bm a) B^{nt}}.
\]
By Proposition \ref{cylinder} (3), we have
\begin{align*}
	q_{n-1} (\bm a) \geq \frac{q_{n-r} (\widetilde{\bm a}) a_{k_1} \cdots a_{k_{r-1}}}{2^{r-1} },
\end{align*}
where $\widetilde{\bm a}$ denotes the sequence obtained from $\bm a$ by removing the words $a_{k_1}, \ldots, a_{k_{r-1}}$ from it.
Therefore,
\begin{align*}
	&	\sum_{\substack{ \bm a \in \mathbb{N}^{n-1}\\
			a_{k_j} \geq B^n, 1 \leq j \leq r-1}}   \frac{1}{q_{n-1}^{2t} (\bm a) B^{nt}} \\
	&\leq \sum_{ \widetilde{\bm a} \in \mathbb{N}^{n-r}} \sum_{a_{k_j} \geq B^n, 1 \leq j \leq r-1} \frac{2^{rt}}{q_{n-r}^{2t} (\widetilde{\bm a}) a_{k_1}^{2t} \cdots a_{k_{r-1}}^{2t} B^{nt}} \\
	&\leq \sum_{\widetilde{\bm a} \in \mathbb{N}^{n-r}} \frac{2^{rt}}{q_{n-r}^{2t} ( \widetilde{\bm a}) B^{n(2t-1)(r-1)} B^{nt}}\\
	&= \sum_{\widetilde{\bm a} \in \mathbb{N}^{n-r}} \frac{2^{rt}}{q_{n-r}^{2t} ( \widetilde{\bm a}) B^{n(t+(2t-1)(r-1))}}.
\end{align*}
Since the number of possible combinations $(k_1,  \ldots, k_{r-1})$ is at most $\tbinom{n}{r}  \leq n^r$ and note that the last summation is independent of our initial choice $k_1<\cdots<k_{r-1}$, it follows that \eqref{eq:tcoverofgamma} is bounded by
\[\begin{split}
	\le&\sum_{ \bm a \in \mathbb{N}^{n-r}} \frac{n^r 2^{rt}}{q_{n-r}^{2t} (\bm a ) B^{n(t+(2t-1)(r-1))} }\\
	\lesssim &n^r  \sum_{ \bm a \in \mathbb{N}^{n-r}} \frac{1}{q_{n-r}^{2t} (\bm a) B^{(n-r)(t+(2t-1)(r-1))} }.
\end{split}\]
Now, the $t$-dimensional Hausdorff measure of $F(r, B)$ can be estimated as follows:
	\begin{align*}
		\mathcal{H}^{t}( F(r,B)) & \leq \liminf_{N \to \infty} \sum_{n=N}^{\infty} n^r  \sum_{ \bm a \in \mathbb{N}^{n-r}} \frac{1}{q_{n-r}^{2t} (\bm a) B^{(n-r)(t+(2t-1)(r-1))} } \\
		&\leq    \liminf_{N \to \infty} \sum_{n=N}^{\infty} n^r e^{-(n-r)\delta}=0,
	\end{align*}
	where the second equality follows from \eqref{eq:sum}.	This implies that $\dim_{\mathrm{H}} F(r, B)\le t$, and hence $\dim_{\mathrm{H}} F(r, B) \leq s(r, B)$.
	Therefore, by \eqref{eq:ub}, we obtain $\dim_{\mathrm{H}} F(r, \psi) \leq s(r, B-\varepsilon)$.
	Letting $\varepsilon \to 0$ and applying the continuity of $s(r, B)$ from Proposition \ref{p:s(B)}, we conclude that
	\[
	\dim_{\mathrm{H}} F(r, \psi) \leq s(r, B). \qedhere
	\]
\end{proof}

To obtain the lower bound, we make use of a result from \cite{Hussain23}. We begin by introducing some notation.
For each $0 \leq i \leq r-1$,
let $\bm \alpha=(A_0, A_1, \ldots, A_{r-1})$, where $A_i >1$ is a real number.
Given an infinite set $\mathbb{A} \subset \mathbb{N}$, define the set
\[
G_r (\bm \alpha, \mathbb{A}):=\big \{x \in [0, 1): c_i A_i^n \leq a_{n+i}(x) < 2c_i A_i^n, 0 \leq i \leq r-1~\text{for i.m.}~n \in \mathbb{A} \big\}
\]
where $c_i \in \mathbb{R}^{+}$.
We denote $G_r(B, \mathbb{A}) :=G_r (\bm \alpha, \mathbb{A})$ when $\bm \alpha=(B, \ldots, B)$, and write $G_r(B)$ for $G_r(B, \mathbb{N})$.
For each $0 \leq i \leq r-1$, define
\[
\beta_{-1} =1, \quad \beta_i= A_0 \cdots A_i,
\]
and let
\[
d_i = \inf \big\{ s \geq 0: P(T, - s \log |T'|-s\log \beta_i +(1-s) \log \beta_{i-1}) \leq 0 \big\}.
\]

\begin{theorem}[{\cite[Theorem 1.1]{Hussain23}}] \label{t:dimension}
Let $\mathbb{A} \subset \mathbb{N}$ be an infinite set. Then we have
	\begin{equation}
	\dim_{\mathrm{H}} G_r(\bm \alpha, \mathbb{A}) =	\dim_{\mathrm{H}} G_r(\bm \alpha) = \min_{0 \leq i \leq r-1} d_i.
	\end{equation}
\end{theorem}

\begin{proof}
In fact, Hussain and Shulga \cite{Hussain23} established only the second equality, but the first can be obtained by slightly modifying their construction of the subset (see, for example, \cite[Corollary 3.3]{Wang2008} for how such a modification can be carried out).
\end{proof}

In order to align with our framework, we set $A_i =B$ and $c_i =B^{r-1}$ for $0 \leq i \leq r-1$.
Under this choice, the set $G_r(B, \mathbb{A})$ and $d_i$ can be rewritten as follows:
\begin{align*}
	G_r (B, \mathbb{A}): &= \left\{x \in [0, 1):  B^{n+r-1} \leq a_{n+i}(x) < 2 B^{n+r-1}, 0 \leq i \leq r~\text{for i.m.}~n \in \mathbb{A} \right\} \\
	&=\big\{x \in [0, 1):  B^{n} \leq a_{n-r+i}(x) < 2 B^{n}, 1 \leq i \leq r~\text{for i.m.}~n  \in \mathbb{A} \big\}
\end{align*}
and
\begin{align*}
d_i &= \inf \big\{ s \geq 0: P(T, -s\log |T'|-s(i+1) \log B+i(1-s)\log B) \leq 0\big\}\\
	&=\inf \big\{ s \geq 0: P(T, -s\log |T'|-(s+(2s-1)i ) \log B) \leq 0\big\}.
\end{align*}
Note that the potential function $-s\log |T'|-(s+(2s-1)i ) \log B$ is decreasing with respect to $i$ since $s >1/2$ (see Proposition \ref{p:s(B)}).
Hence, the minimum of $d_i$ is attained at $i=r-1$.
By Theorem \ref{t:dimension}, it follows that
\begin{align*}
	\dim_{\mathrm{H}} G_r (B, \mathbb{A})&= \min_{0 \leq i \leq r-1} d_i = d_{r-1}\\
	&= \inf \big\{ s \geq 0: P(T, -s \log |T'|-(s+(2s-1)(r-1))\log B) \leq 0 \big\},
\end{align*}
which shows that $\dim_{\mathrm{H}} G_r(B, \mathbb{A}) =s(r, B)$.

We now proceed to establish the lower bound for the Hausdorff dimension of the set $F(r, \psi)$.
Recall that
\[
\log B = \liminf_{n \to \infty} \frac{\log \psi(n)}{n}.
\]
Then, for any $\varepsilon >0$,
the inequality $\psi(n) \leq (B+\varepsilon)^n$ holds for infinitely many $n \in \mathbb{N}$.
Denote
\[
\mathbb{A}: =\{ n \in \mathbb{N}: \psi(n) \leq (B +\varepsilon)^n\},
\]
which is clearly an infinite set.

It is obvious that $G_r (B+\varepsilon, \mathbb{A})  \subset F(r, \psi)$, and then
\begin{align*}
\dim_{\mathrm{H}} F(r, \psi) 	\geq \dim_{\mathrm{H}} G_r (B+\varepsilon, \mathbb{A})
	= s(r, B+\varepsilon).
\end{align*}
Taking the limit as $\varepsilon \to 0$, we conclude that
\[
\dim_{\mathrm{H}} F(r, \psi) \geq  s(r, B).
\]

\subsection{The Case $B=1$}\

For any $\delta >1$, let $\psi_1(n) = \delta^n \cdot \psi(n)$.
Then we have
\[
\liminf_{n \to \infty} \frac{\log \psi_1(n)}{n} =\log \delta >0
\]
and
\[
F(r, \psi_1) \subset F(r, \psi).
\]
By the case $1 < B < \infty$, we know that
\[
\dim_{\mathrm{H}} F(r, \psi_1) = s(r, \delta)
\]
where $s(r, \delta)$ is defined in Section 2.3.
Let $\delta \to 1$ and by Proposition \ref{p:s(B)}, we obtain
\[
  \dim_{\mathrm{H}} F(r, \psi_1) = 1.
\]
Hence,
\[
\dim_{\mathrm{H}} F(r, \psi) \geq   \dim_{\mathrm{H}} F(r, \psi_1) = 1.
\]

\subsection{The Case $B=\infty$}\

In this case, we will make use of the following result, due to {\L}uczak \cite{Luczak1997} and Feng, Wu, Liang, and Tseng \cite{Feng1997}.

\begin{theorem}[\cite{Feng1997, Luczak1997}] \label{thm:dim}
	For any $b, c > 1$, both of the following two sets
	\[
	\big\{ x \in [0,1): a_n(x) \geq c^{b^n} \text{for~i.m.}~ n \big\}
	\]
	and
	\[
	\big\{ x \in [0,1): a_n(x) \geq c^{b^n} \text{for all}~  n  \gg 1 \big\}
	\]
	have the same Hausdorff dimension $\dfrac{1}{1+b}$.
\end{theorem}
To apply this result to our setting, we recall that
\[
~\log b =  \liminf_{n \to \infty} \frac{\log \log  \psi(n) }{n}.
\]
Based on the value of $b$, we divide the discussion into three cases.

{\bf Case 1:} $b=1$.
For any $A$ with $1 <A< \infty$, we have $\psi(n) \geq \psi_A(n):=A^n$ since $B= \infty$.
Therefore, it is clear that $F(r, \psi) \subset F(r, A)$ and from the case $1<B< \infty$, we know that
\[
\dim_{\mathrm{H}} F(r, A)=s(r, A).
\]
This implies that
\[
\dim_{\mathrm{H}} F(r, \psi) \leq s(r,  A).
\]
Letting $A \to \infty$, and using Proposition \ref{p:s(B)} (2), we obtain
\[
\dim_{\mathrm{H}} F(r, \psi) \leq \frac{1}{2}.
\]

On the other hand, for any $d>1$,
by the definition of $b$,  one has
\[
\psi(n) \leq e^{d^n} ~\text{for~infinitely~many}~n.
\]
Therefore,
\[
\{x \in [0, 1): a_n(x) \geq e^{d
	^{n}}~\text{for~all}~n \gg 1\} \subset F(r, \psi),
\]
and by Theorem \ref{thm:dim},
\[
\dim_{\mathrm{H}} F(r, \psi) \geq \dfrac{1}{1+d},
\]
Letting $d \to 1$, we obtain $\dim_{\mathrm{H}} F(r, \psi) \geq \dfrac{1}{2}$.

{\bf Case 2:} $1 < b < \infty$.
For any $1<c < b$, by the definition of $b$, we have
\[
\psi(n) \geq  e^{c^n}~\text{ for all } n \gg 1.
\]
Therefore, by \eqref{eq:limsup}, we have
\[
F(r, \psi ) \subset \{x \in [0, 1): a_n(x) \geq e^{c^n}~\text{for i.m.}~ n \in \mathbb{N}\}
\]
which, together with Theorem \ref{thm:dim}, implies
\[
\dim_{\mathrm{H}} F(r, \psi) \leq \dfrac{1}{1+c}.
\]
Letting $c \to b$, we obtain
\[
\dim_{\mathrm{H}} F(r, \psi) \leq \dfrac{1}{1+b}.
\]

On the other hand, for any $d > b$, by the definition of $b$, one has
\[
\psi(n) \leq e^{d^n}~\text{for~infinitely~many}~n.
\]
Hence,
\[
\{x \in [0, 1): a_n(x) \geq e^{d^{n}}~\text{for~all}~n \gg 1\} \subset F(r, \psi),
\]
 and by Theorem \ref{thm:dim},
\[
\dim_{\mathrm{H}} F(r, \psi) \geq \dfrac{1}{1+d}.
\]
Letting $d \to b$ yields
\[
\dim_{\mathrm{H}} F(r, \psi) \geq \dfrac{1}{1+b}.
\]

{\bf Case 3:} $b =\infty$.
For any $1<c < b$, by the definition of $b$, we have
\[
\psi(n) \geq  e^{c^n}~\text{ for all } n \gg 1.
\]
Therefore,
\[
F(r, \psi ) \subset \{x \in [0, 1): a_n(x) \geq e^{c^n}~\text{for i.m.}~ n \in \mathbb{N}\}.
\]
Together with Theorem \ref{thm:dim}, this implies
\[
\dim_{\mathrm{H}} F(r, \psi) \leq \dfrac{1}{1+c}.
\]
Letting $c \to \infty$, we obtain
\[
\dim_{\mathrm{H}} F(r, \psi) =0.
\]

{\bf Acknowledgement}  The work was supported by the Fundamental Research Funds for the Central Universities (No. SWU-KQ24025).

\bibliographystyle{amsplain}

\begin{thebibliography}{10}



\bibitem{Bernstein1911}
F. Bertein, {\it \"Uber eine Anwendung der Mengenlehre auf ein aus der Theorie der s\"akularen St\"orungen herr\"uhrendes problem}, Math. Ann. 71 (1911), 417--439.

\bibitem{Borel1912}
E. Borel, {\it Sur un probleme de probabilit\'es relatif aux fractions continnues}, Math. Ann. 72 (1912), 578--584.


\bibitem{Chung52}
K. L. Chung and P. Erd\"os, {\it On the application of the Borel-Cantelli lemma}, Trans. Amer. Math. Soc. 72 (1952), 179--186.

\bibitem{Diamond1986}
H. Diamond and J. Vaaler, {\it Estimates for partial sums of continued fraction partial quotients}, Pac. J. Math. 122 (1986), 73--82.


\bibitem{Dolgopyat22}
D. Dolgopyat, B. Fayad and S. Liu,  {\it Multiple Borel-Cantelli lemma in dynamics and multilog law for recurrence}, J. Mod. Dyn. 18 (2022), 209--289.


\bibitem{Falconer2013}
K. Falconer, {\it Fractal Geometry: Mathematical Foundations and Applications}, John Wiley \& Sons Ltd., 3rd edition (2013).

\bibitem{Feng1997}
D. J. Feng, J. Wu, J. C. Liang and S. Tseng,  {\it Appendix to the paper by T.  Luczak--a simple proof of the lower bound: `` On the fractional dimension of sets of continued fractions"}, Mathematika 44 (1997), 54--55.



\bibitem{Huang20}
L.  Huang, J. Wu and J. Xu,  {\it Metric properties of the product of consecutive partial quotients in continued fractions}, Israel J. Math. 238 (2020), no. 2, 901--943.

\bibitem{Hanus2002}
P. Hanus, D. Mauldin and M. Urbanski,  {\it Thermodynamic formalism and multifractal analysis of conformal infinite iterated function systems}, Acta Math. Hungar. 96 (2002), 27--98.


\bibitem{Hussain23}
M. Hussain and N. Shulga, {\it Metrical properties of exponentially growing partial quotients}, arXiv: 2309. 10529.


\bibitem {Iosifescu02}
M. Iosifescu and C. Kraaikamp, {\it Metrical theory of continued fractions}, Mathematics and its Applications, 547. Kluwer Academic Publishers, Dordrecht, 2002.


\bibitem{Khintchine63}
A. Ya Khintchine, {\it Continued fraction}, P. Noordhoff, Groningen, The Nethlands, 1963.

\bibitem{Kleinbock2018}
D. Kleinbock and N. Wadleigh, {\it A zero-one law  for improvements to Dirichlet's Theorem}, Proc. Amer. Math. Soc. 146 (2018), 1833--1844.

\bibitem{Li2014}
B. Li, B. W. Wang, J. Wu and J. Xu, {\it The shrinking target problem in the dynamical system of continued fraction}, Proc. Lond. Math. Soc. (3) 108 (2014), no. 1, 159--186.




\bibitem{Luczak1997}
T. {\L}uczak,  {\it  On the fractional dimension of sets of continued fractions}, Mathematika 44(1997), 50--53.

\bibitem{MU1996}
D. Mauldin and M. Urba\'nski, {\it Dimensions and measures in infinite iterated function systems}, Proc. London Math. Soc. (3) 73 (1996), no. 1, 105--154.


\bibitem{MU1999}
D. Mauldin and M. Urba\'nski, {\it Conformal iterated function systems with applications to the geometry of continued fractions}, Trans. Amer. Math. Soc. 351 (1999), no. 12, 4995--5025.

\bibitem{MU2003}
D. Mauldin and M. Urba\'nski, {\it   Graph directed Markov systems. Geometry and dynamics of limit sets}, Cambridge Tracts in Mathematics, 148. Cambridge University Press, Cambridge, 2003.

\bibitem{Petrov02}
V. Petrov, {\it A note on the Borel--Cantelli lemma}, Statist. Probab. Lett. 58 (2002), 283--286.


\bibitem{Philipp67}
W. Philipp, {\it Some metrical theorems in number theory}, Pacific J. Math. 20 (1967), 109--127.

\bibitem{Philipp88}
W. Philipp, {\it Limit theorems for sums of partial quotients of continued fractions}, Monatsh. Math. 105 (1988), no. 3, 195--206.


\bibitem{Ruelle1973}
D. Ruelle, {\it Statistical mechanics on a compact set with $\mathbb{Z}^v$ action satisfying expansiveness and specification}, Trans. Amer. Math. Soc. 187 (1973), 237--251.


\bibitem{Tan23}
B. Tan, C. Tian and B. W. Wang, {\it The distribution of the large partial quotients in continued fraction expansions}, Sci. China Math. 66 (2023), no. 5, 935--956.


\bibitem{Tan24}
B. Tan and Q. L. Zhou, {\it  Metrical properties of the large products of partial quotients in continued fractions}, Nonlinearity 37 (2024), no. 2, Paper No. 025008, 28 pp.

\bibitem{Walters1975}
P. Walters, {\it A variational principle for the pressure of continuous transformations}, Amer. J. Math. 97(1975), 937--971.

\bibitem{Walters1982}
P. Walters, {\it An introduction to ergodic theory}, volume 79 of Graduate Texts in Mathematics, Springer-Verlag, New York-Berlin, 1982.


\bibitem{Wang2008}
B. W. Wang and J. Wu, {\it Hausdorff dimension of certain sets arising in continued fraction expansions}, Adv. Math. 218 (2008), no. 5, 1319--1339.


\bibitem{Wu2006}
J. Wu, {\it A remark on the growth of the denominators of convergents}, Monatsh. Math. 147 (2006), no. 3, 259--264.













\end{thebibliography}

\end{document}